\documentclass[11pt,a4paper,reqno]{amsart}
\usepackage[utf8]{inputenc}
\usepackage[leqno]{amsmath}

\usepackage[top=2.3cm,bottom=2.3cm,left=1.9cm,right=1.9cm]{geometry}

\usepackage{amsfonts,amsmath,amssymb,amsthm,dsfont,amsxtra,enumerate}
\usepackage{graphicx}
\usepackage[dvipsnames]{xcolor}
\usepackage{hyperref}
\usepackage{multicol}
\usepackage{mathtools}
\usepackage{listings}
\usepackage{pgf,tikz}
\usepackage{comment}

\newtheorem{thm}{Theorem}[section]
\newtheorem{cor}[thm]{Corollary}
\newtheorem{lemma}[thm]{Lemma}
\newtheorem{prop}[thm]{Proposition}

\newtheorem{definition}[thm]{Definition}
\theoremstyle{definition}
\newtheorem{remark}[thm]{Remark}
\newtheorem{eg}[thm]{Example}
\newtheorem{notation}[thm]{Notation}
\numberwithin{equation}{section}
\newtheorem{question}{Question}[section]

\def\NN{\mathbb{N}}

\DeclareMathOperator{\reg}{reg} 
\DeclareMathOperator{\het}{ht}

\DeclareMathOperator{\pmd}{pmd}
\DeclareMathOperator{\Tor}{Tor}
\DeclareMathOperator{\tpmd}{tpmd}

\begin{document}

\title[(Almost) Complete Intersection LSS-ideals and regularity of their powers]{(Almost) Complete Intersection Lov\'{a}sz$-${S}aks$-${S}chrijver ideals and regularity of their powers}
\author[Marie Amalore Nambi]{Marie Amalore Nambi}
\address{Department of Mathematics, Indian Institute of Technology Hyderabad, Kandi, Sangareddy - $502285$, India}
\email{amalore.p@gmail.com}

\author[Neeraj Kumar]{Neeraj Kumar}
%\address{Department of Mathematics, Indian Institute of Technology Hyderabad, Kandi, Sangareddy - $502285$, India}
\email{neeraj@math.iith.ac.in}

\author[Chitra Venugopal]{Chitra Venugopal}
%\address{Department of Mathematics, Indian Institute of Technology Hyderabad, Kandi, Sangareddy - $502285$, India}
\email{ma19resch11002@iith.ac.in}

%13F65 Commutative rings defined by binomial ideals, toric rings, etc
%13F70 Other commutative rings defined by combinatorial properties
% 13C40 Linkage, complete intersections and determinantal ideals
%14M10 Complete intersections 
% 13D02 Syzygies, resolutions, complexes and commutative rings
%05E40 Combinatorial aspects of commutative algebra

\date{}

\begin{abstract}
We discuss the property of (almost) complete intersection of LSS-ideals of graphs of some special forms, like trees, unicyclic, and bicyclic graphs. Further, we give a sufficient condition for the complete intersection property of twisted LSS-ideals in terms of a new graph theoretical invariant called twisted positive matching decomposition number denoted by $\tpmd$. 
\end{abstract}

\maketitle

\section*{Introduction}

Let $G$ be a graph on $[n]=\{1,\ldots,n \}$, $k$ be a field and $d \geq 1$ be an integer. The Lov\'{a}sz$-$ {S}aks$-${S}chrijver ideal (cf. \cite{HMSW, LSS89}) in the polynomial ring $S=k[x_{ij} \mid i \in [n], j \in [d]]$ is defined as $$L_{G}^k(d)= ( f_e^{(d)}= \sum_{\ell=1}^{d}x_{i\ell}x_{j\ell} \mid e= \{i,j\} \in E(G)).$$
We refer to it as LSS-ideal in short and denote it by $L_{G}(d)$ when the field $k$ is evident. It defines the variety of orthogonal representations of the complementary graph of $G$ (cf. \cite{HMSW, LSS89}). The ideal  $L_G(1)$ coincides with the edge ideal of a graph $G$. Another way of looking at the LSS-ideals is as the defining ideal of a symmetric algebra since the generators of $L_G(d)$ have a degree at most $1$ in each variable. 

The twisted LSS-ideal (cf. \cite{CW2019}, p.475) in the polynomial ring $S=k[x_{ij} \mid i \in [n], j \in [2d]]$ is defined as 
$$\hat{L}_{G}^k(d)= ( \hat{f}_e^{(d)}= \sum_{\ell=1}^{d}x_{i2\ell -1}x_{j2\ell}-x_{i2\ell }x_{j2\ell-1} \mid e= \{i,j\} \in E(G)).$$
We denote it by $\hat{L}_{G}(d)$ when the field $k$ is evident. Clearly, the ideal $\hat{L}_{G}(1)$ coincides with the binomial edge ideal of a graph $G$ (cf. \cite{HH, O2011}).

In \cite{CW2019}, the authors show that $\hat{L}_G(d)$ is isomorphic to $L_G(2d)$ when $G$ is bipartite and hence the algebraic properties of being prime and radical transforms from one to the other for all $d$.

\vspace{2mm}

LSS-ideals and twisted LSS-ideals have a close relationship with some classes of ideals associated with graphs, like the determinantal ideals of the $(d+1)$-minors of generic/ generic symmetric matrices with $0$s in positions corresponding to the edges of graph $G$ (denoted by $X^{gen}_{G}/X^{sym}_{G}$ respectively) and Pfaffian ideals of order $2d$ of generic skew-symmetric matrices with entries prescribed by the edges of $G$ (denoted by $X^{skew}_{G}$) respectively. This is evident from Remark \ref{Isomorphisms}, which involves isomorphisms discussed in \cite{CW2019}. 

\vspace{2mm}
 
 An ideal $I$ of a ring $R$ is said to be a complete intersection if the minimal number of generators of $I$ is equal to its height. In \cite{CW2019}, the authors introduce a graph theoretical invariant called the positive matching decomposition number denoted by $\pmd$ (see Section \ref{section 1} for definition), which helps in the study of complete intersection property of LSS-ideals. In fact, the authors prove the following implications.
 $$ d \geq \pmd(G) \Rightarrow L_{G}(d) \text{ is a radical complete intersection} \Rightarrow L_G(d+1) \text{ is prime}.$$

Since the primality of the LSS-ideals implies the irreducibility of the corresponding variety of orthogonal representations, the study of $\pmd$ of graphs have applications in both algebra and geometry. More results on the $\pmd$ of graphs and hypergraphs are given in \cite{ANpmd, FG2022, GW2023}.

\vspace{2mm}

Algebraic properties of the ideal $L_G(2)$, such as primary decomposition, radical, prime, and complete intersections are studied in terms of the combinatorial invariants of $G$ in \cite{CW2019, HMSW, AK2021}. 
In \cite{CW2019}, Conca and Welker try to answer the questions: When is $L_G(d)$ radical, prime, and complete intersection? In this direction, the authors completely characterize the above properties when $G$ is a forest, in terms of the maximal degree of vertices in $G$ and $d$ (see Remark \ref{R,CI,Prime}).

\vspace{2mm}

In 
Section \ref{section 2}, we characterize complete intersection LSS-ideals corresponding to unicyclic and bicyclic graphs. Here $\Delta(G)$ is the maximal degree of the vertices in $G$.
\begin{thm} \label{CI_LSS}
    Let $G$ be a graph and $d$ be a positive integer. 
    \begin{enumerate}[(a)]
        \item \label{thm1.1} If $G$ is unicyclic with $d\geq 3$, then $L_G(d)$ is a complete intersection if and only if $d \geq \Delta(G)$.
    
        \item \label{thm1.3} If $G$ is bicyclic with $d \geq 4$, then $L_{G}(d)$ is a complete intersection if and only if $d \geq \Delta(G)$.

    \end{enumerate}    
\end{thm}

In Section \ref{section 2}, we also define a graph theoretical invariant, called the twisted positive matching decomposition number ($\tpmd$), similar to $\pmd$, which helps in the study of twisted LSS-ideals. We give a sufficient condition for the twisted LSS-ideals to be radical complete intersections in terms of this new invariant. That is, we prove the following implication (see Theorem \ref{CI_pmd}).
$$ d \geq \tpmd(G) \Rightarrow \hat{L}_{G}(d) \text{ is a radical complete intersection}.$$

\vspace{2mm}

An ideal $I$ of a ring $R$ is said to be an almost complete intersection if the minimal number of generators of $I$ is one more than the height of $I$ along with the property that for all minimal primes $\mathfrak{p}$ of $I$ in $A$, $I_{\mathfrak{p}}$ is a complete intersection. In \cite{AK2021}, for $d=2$, Kumar characterizes graphs whose LSS-ideals are almost complete intersections and studies Cohen-Macaulayness of the Rees algebra of almost complete intersection LSS-ideals. 
In this article, for all $d\geq 2$, we characterize almost complete intersection LSS-ideals corresponding to trees and $C_3$-free connected unicyclic graphs (see Theorems \ref{ACI:tree}, and \ref{ACI:unicyclic}).

\vspace{2mm} 
 
Castelnuovo-Mumford regularity is an important algebraic invariant which measures the complexity of modules. 
It is well known that for any homogeneous ideal $I$, the regularity of $I^s$ is of the form $as+b$ for $s \gg 0$ and $a,b$ being non-negative constants, see \cite{CHT1999, V2000}. The value of $a$ is well understood in literature, whereas computing $b$ is found to be a difficult problem in general. For some classes of graph $G$, the regularity of binomial ideals ($L_G(2)$) and their powers are studied in \cite{AN,ANcycle,VRT2021,JAR2019,A2021reg,MM,SZ}. It is observed that characterizing complete intersections and almost complete intersections helps in the study of the regularity of ideals and their powers.  
 In Proposition \ref{rem:CIreg}, we obtain the constant $b$ for LSS-ideals of trees and unicyclic graphs in terms of the number of vertices of the graph.  
%For forests and unicyclic graphs with $\Delta(G)\leq d$, the regularity of powers of the corresponding LSS-ideals is computed, see Proposition \ref{rem:CIreg}. 
We give a lower bound for the regularity of powers of LSS-ideals in terms of certain invariants corresponding to its induced subgraphs, see Proposition \ref{pro:ij}. Also, we obtain bounds for the regularity of powers of almost complete intersection LSS-ideals associated with trees, unicyclic and bicyclic graphs, see Theorems \ref{thm4:t}, and \ref{thm4:ub}. %Some other algebraic properties looked into are the Koszulness of the quotients of LSS-ideals. 

\vspace{4mm}

 \noindent {\bf Acknowledgement.} We thank the anonymous referee for their valuable comments and suggestions. The first author is financially supported by the University Grant Commission, India. The Core Research Grant (CRG/2023/007668) from the Science and Engineering Research Board, ANRF, India, partially supports the second author. The third author is financially supported by INSPIRE fellowship, DST, India.

\section{Preliminaries} \label{section 1}
Throughout the article, unless otherwise stated, $d$ denotes a positive integer, and $G$ a finite simple undirected graph with vertex set $V(G)=[n]$ and edge set $E(G)$.

\vspace{2mm}

\noindent {\textbf{Definitions and Notations.}}
\begin{itemize}
    \item A \textit{subgraph} of $G$ is a graph $H$ such that $V(H) \subset V(G)$ and $E(H) \subset E(G)$.
    \item For $U \subset V(G)$, $G[U]$ denotes the \textit{induced subgraph} of $G$ on vertex set $U$. For $i,j \in U$, $\{i,j\} \in E(G[U])$ if and only if $\{i,j\} \in E(G)$. For a vertex $u \in V(G)$, $G \setminus u$ denotes the induced subgraph on $V(G) \setminus u$.
    \item For $m,n >0$, $K_n$ denotes the \textit{complete} graph on $[n]$. $K_{m,n}$ denotes the \textit{complete bipartite} graph on $[m+n]$. For $n>2$, $C_n$ denotes the \textit{cycle} on $[n]$.
    \item A graph $G$ is said to be a \textit{forest} if it does not have a cycle as a subgraph and a \textit{tree} if it is connected.
    \item A graph $G$ is said to be a \textit{unicyclic graph} if $G$ contains precisely one cycle as a subgraph.
    \item A graph $G$ is said to be a \textit{bicyclic graph} if $G$ contains exactly two cycles as a subgraph.
    \item For a vertex $v \in V(G)$, \textit{degree} of a vertex, denoted by $deg_{G}(v)$, is the number of edges incident to $v$.
    \item For a graph $G$, $\Delta(G)= \max_{v\in V(G)} \deg_G(v)$. 
\end{itemize}

\vspace{2mm}

The following remark discusses the relations between the LSS-ideals and the twisted LSS-ideals with the determinantal ideals and the Pfaffian ideals, respectively.  

\begin{remark} \label{Isomorphisms}
For a matrix $X$ with variables as entries, let $I_l(X)$ denote the ideal of $K[X]$ generated by the $l$-minors of $X$, and for a generic skew-symmetric matrix $X'$, let Pf$_{l}(X')$ denote the Pfaffian ideal of order $l$ in $K[X']$ which is generated by the square roots of the determinants of the submatrices of $X'$ obtained by considering its $l$-rows and the corresponding $l$-columns.
    \begin{enumerate} [(a)]
    \item \label{iso1} Let $G$ be a subgraph of a complete bipartite graph $K_{m,n}$ where $m,n \in \NN$, then $K[x_{ij}]/(I_{d+1}(X^{gen}_{G})+(x_{ij} \, | \{i,j\} \in E)) \cong K[YZ]/L_G(d) \cap K[YZ]$ where $Y=(y_{ij})$ and $Z=(z_{ij})$ are $m \times d$ and $d \times n$ matrices of variables respectively.
    \item \label{iso2} Let $G$ be a subgraph of a complete graph $K_{n}$ where $n \in \NN$, then $K[x_{ij}]/(I_{d+1}(X^{sym}_{G})+(x_{ij} \, | \{i,j\} \in E)) \cong K[YY^{T}]/L_G(d) \cap K[YY^{T}]$ where $Y=(y_{ij})$ is an $n \times n$ matrix of variables.
    \item \label{iso3} Let $G$ be a subgraph of a complete graph $K_{n}$ where $n \in \NN$ and for $\hat{f}_e^{(d)}=\sum_{k=1}^d(y_{i\, 2k-1}y_{j \, 2k}-y_{i \, 2k}y_{j \, 2k-1})$, let $\hat{L}_G(d)=\{ \hat{f}_e^{(d)}: \, e \in E \}$ be the twisted LSS-ideal associated to $G$. 
    Then $$K[x_{ij}]/(\text{Pf}_{2d+2}(X^{skew}_{G})+(x_{ij} \, | \{i,j\} \in E)) \cong K[YJY^{T}]/ \hat{L}_G(d) \cap K[YJY^{T}]$$ where $Y=(y_{ij})$ is an $n \times 2d$ matrix of variables and $J$ is a $2d \times 2d$ block matrix with $d$ blocks of $\begin{pmatrix}
                 0 & 1 \\
                -1 & 0 \\
        \end{pmatrix}$ on the diagonal and $0$ in the remaining positions. 
        %Note that for $d=1$, $\hat{L}_G(d)$ coincides with certain binomial edge ideals, and for some graphs like bipartite graphs, properties of being prime and radical transfers from    $\hat{L}_G(d)$ to $L_G(2d)$ respectively for all $d$. 
\end{enumerate}
\end{remark}

\vspace{2mm}

The property of LSS ideals corresponding to forests being radical, prime, and complete intersections is mentioned in the following remark, which is referred to repeatedly in this article.

\begin{remark} \cite[Theorem 1.5]{CW2019} \label{R,CI,Prime} Let $G$ be a forest and denote by $\Delta(G)$ the maximal degree of a vertex in $G$. Then
\begin{enumerate}[(a)]
    \item $L_G(d)$ is radical for all $d$.
    \item $L_G(d)$ is a complete intersection if and only if $d \geq \Delta(G)$.
    \item $L_G(d)$ is prime if and only if $d \geq \Delta(G) +1$.
\end{enumerate}
\end{remark}

We now recall the definition of the positive matching decomposition number, as introduced by  Conca and Welker \cite{CW2019}. 

\begin{definition}
Let $G=(V(G),E(G))$ be a graph. A subset $M \subseteq E(G)$ is said to be a matching in $G$  if the edges in $M$ are pairwise disjoint. A matching decomposition of $G$ is a partition of the edge set $E(G)=\cup_{i=1}^{p} M_i$ into pairwise disjoint subsets, where each $M_i$ is a matching in $G$ for $i = 1,\ldots,p$. 
\end{definition}

\begin{definition} \cite[Definition 5.1]{CW2019}
    Given a graph $G=(V(G),E(G))$ a positive matching of $G$ is a subset $M \subseteq E(G)$ of pairwise disjoint sets such that there exists a weight function $w : V(G) \rightarrow \mathbb{R}$ satisfying:
    \begin{equation*}
      \sum_{i\in e} w(i)>0 \text{ if } e \in M, \hspace{2cm} \sum_{i\in e} w(i)<0 \text{ if } e \in E\setminus M. 
    \end{equation*}
\end{definition}

%In \cite{CW2019}, the authors introduce a graph theoretical invariant called the positive matching decomposition number and show it to be related to the algebraic properties of LSS-ideals. Further study on positive matching decompositions ($\pmd$) of graphs is done in \cite{FG2022}. It is defined as follows.

\begin{definition} \cite[Definition 5.3]{CW2019}
Let $G=(V(G),E(G))$ be a graph. A positive matching decomposition of $G$ is a partition $E(G) = \cup_{i=1}^{p} M_i$ into pairwise disjoint subsets such that $M_i$ is a positive matching on $(V(G), E(G) \setminus \cup_{j=1}^{i-1} M_j)$ for $i = 1,\ldots,p$. 
%We denote a positive matching decomposition by pm-decomposition and $M_i$ are called the parts of the pm-decomposition.
The smallest $p$ for which $G$ admits a pm-decomposition with $p$ parts will be denoted by $\pmd(G)$.
\end{definition}

Important properties of the $\pmd$ of graphs, along with the characterization of a positive matching, are recalled below.

\begin{remark} \cite[Theorem 1.3]{CW2019} \label{pmdCI}
    Let $G$ be a graph. Then
 for $d \geq \pmd(G)$, the ideal $L_G(d)$ is a radical complete intersection. In particular, $L_G(d)$ is prime if $d \geq \pmd(G) + 1$.
\end{remark}

\begin{remark} \cite[Lemma 5.4.(3)]{CW2019} \label{rem.pmdforest}
    Let $G$ be a graph. Then $\pmd(G) \geq \Delta(G)$ and attains equality if $G$ is a forest.
\end{remark}

\begin{remark} \cite[Theorem 2.1]{FG2022} \label{pos.match}
    Let $G$ be a graph. A matching $M$ of a graph $G$ is positive if and only if the subgraph of $G$ induced by $M$ has no alternating closed walks with respect to $M$.
\end{remark}

In the following, we recall basic definitions and results from commutative algebra.

\begin{remark} \cite[Lemma 2.2]{reg-seq} \label{rem.reg.seq}
    Let $R$ be a commutative ring equipped with a term order $<$. Let $I$ be an ideal of $R$. If the initial ideal $\textnormal{in}_{<}(I)$ is radical, a complete intersection or prime, then $I$ shares the same property.
\end{remark}

\begin{remark}\cite[Lemma 4.1]{AK2021} \label{colonp} Let $I$ be a radical ideal in a Noetherian commutative ring $R$. Then, for any
$f \in R$ and $n \geq 2$, $$I : f = I : f^n.$$
\end{remark}

\begin{remark}\cite[Lemma 4.2]{AK2021} \label{CIorACI} If $I$ is a homogeneous ideal in a polynomial ring such that $I = J + (a)$, where $J$ is generated by a homogeneous regular sequence, $a$ is a homogeneous element and
$J: a = J: a^2$, then $I$ is either a complete intersection or an almost complete intersection.
\end{remark}

\begin{definition}
    Let $S$ be a standard graded polynomial ring over a field $k$ and 
    $M$ a finitely generated graded $S$-module. 
    Then the \emph{Castelnuovo-Mumford regularity} or simply  \emph{regularity} of $M$ over $S$, denoted by $\reg_S (M)$, is defined as
    $$\reg_S(M) = \max \{j-i \mid \Tor_i^S(M,k)_{j} \neq 0\}.$$
    For convenience, we shall use $\reg(M)$ instead of $\reg_S(M)$.
\end{definition}

\begin{remark}\cite[Lemma 4.4]{BHT15} \label{regular}
    Let $u_1,\ldots,u_n$ be a regular sequence of homogeneous polynomials in $S$ with $\deg(u_i)=d$. Let $I= (u_1,\ldots,u_n)$ be an ideal. Then for all $s \geq 1$, we have 
    $$\reg(I^s)=ds+(d-1)(n-1).$$
\end{remark}

\begin{remark} \cite[Corollary 2.11]{JAR2019} \label{Reg.upper}
     Let $S$ be a standard graded polynomial ring over a field $k$ and $u_1,\ldots,u_n$ be a homogeneous $d$-sequence with $\deg(u_i)= d_i$ in $S$ such that $u_1,\ldots, u_{n-1}$ is a regular sequence. Set $I=(u_1,\ldots,u_n)$ and $d=\max\{d_i : 1 \leq i \leq n\}.$ Then, for all $s \geq 1$, 
    $$\reg(S/I^s) \leq d(s-1) + \max\{\reg(S/I),\sum_{i=1}^{n-1}d_i-n\}.$$
\end{remark}

\section{Complete Intersection} \label{section 2}

We begin this section by introducing the notion of a twisted positive matching decomposition of a graph $G$ and give a sufficient condition for the complete intersection property of twisted LSS-ideals in terms of this graph theoretical invariant.

\vspace{2mm}

\begin{definition} \label{def.tmd}
    Let $G=(V(G),E(G))$ be a graph and $p$ be a positive integer. 
    \begin{enumerate}
        \item[(a)] A twisted matching decomposition of $G$ is a partition $E(G)=\cup_{\ell=1}^{2p}M_\ell$ into pairwise disjoint subsets such that 
\begin{equation} \label{eq.matchcond}
    \text{ if } \{i,j\}\in M_{2q-1} \text{ then } \{k_1,i\} \text{ and } \{j,k_2\} \notin M_{2q}
\end{equation} 
for all $q=1,\ldots,p$, where $i,j,k_1,k_2 \in V(G)$ with $k_1<i<j<k_2$, and $M_{2q}$ can be empty set.
\item[(b)] Let $E(G)=\cup_{\ell=1}^{2p}M_\ell$ be a twisted matching decomposition of $G$. For $q=1,\ldots,p$,  let $H_q$ be a graph corresponding to the pair $(M_{2q-1},M_{2q})$, where the vertex set and the edge set are 
$$V(H_q)=\{1_{2q-1},2_{2q-1},\ldots,n_{2q-1},1_{2q},2_{2q},\ldots,n_{2q}\} \text{ and }
E(H_q)=\begin{cases}
    \{i_{2q-1},j_{2q}\} & \text{ if } \{i,j\}\in M_{2q-1}\\
    \{j_{2q-1},i_{2q}\} & \text{ if } \{i,j\}\in M_{2q},\\
\end{cases}$$
where $i,j \in V(G)(i<j)$.
    \end{enumerate} 
\end{definition}

Observe that, for a graph $G$, for all $i,j \in V(G)$ $(i \neq j)$, if $\{i_{2q-1},j_{2q}\} \in E(H_q)$ then $\{j_{2q-1},i_{2q}\} \notin E(H_q)$, and $\{i_{2q-1},i_{2q}\} \notin E(H_q)$.

\begin{eg} \label{eg.hq}
    Let $G$ be the graph as given in Figure \ref{fig.sq}. We present a twisted matching decomposition for the graph $G$. If $M_1=\{\{1,2\}\}$, $M_2=\{\{1,4\}\}$, $M_3=\{\{2,3\}\}$, and $M_4=\{\{1,3\},\{2,4\}\}$ then $E(G)=M_1 \cup \cdots \cup M_4$ is a twisted matching decomposition of $G$. Note that the edge $\{2,3\}$ cannot be included in $M_2$, even though  $\{\{1,4\},\{2,3\}\}$ is a matching in $G$, as dictated by Equation (\ref{eq.matchcond}).

\begin{figure}[ht]
    \centering

\tikzset{every picture/.style={line width=0.75pt}} %set default line width to 0.75pt        

\begin{tikzpicture}[x=0.75pt,y=0.75pt,yscale=-1,xscale=1]
%uncomment if require: \path (0,771); %set diagram left start at 0, and has height of 771

%Straight Lines [id:da3155723752048024] 
\draw    (190,323) -- (190.5,408.44) -- (283.5,408.44) ;
%Straight Lines [id:da4983952286361343] 
\draw    (190,323) -- (284.5,324.44) ;
%Straight Lines [id:da19211380368200826] 
\draw    (284.5,324.44) -- (190.5,408.44) ;
%Straight Lines [id:da24711894422196878] 
\draw    (191.5,324.44) -- (283.5,408.44) ;

% Text Node
\draw (173,306) node [anchor=north west][inner sep=0.75pt]  [font=\small]  {$1$};
% Text Node
\draw (173,404) node [anchor=north west][inner sep=0.75pt]  [font=\small]  {$2$};
% Text Node
\draw (283.5,408.44) node [anchor=north west][inner sep=0.75pt]  [font=\small]  {$3$};
% Text Node
\draw (288,307) node [anchor=north west][inner sep=0.75pt]  [font=\small]  {$4$};
% Text Node
%\draw (223.5,437.44) node [anchor=north west][inner sep=0.75pt]  [font=\small]  {$G$};

\filldraw[black] (190,323) circle (1.5pt) ;
\filldraw[black] (284.5,324.44) circle (1.5pt); 
\filldraw[black] (283.5,408.44) circle (1.5pt); 
\filldraw[black] (190.5,408.44) circle (1.5pt); 
\end{tikzpicture}
\caption{The graph $G$}
\label{fig.sq}
\end{figure}

Below, we present the graphs $H_1$ and $H_2$ corresponding to the pair $(M_1,M_2)$ and $(M_3,M_4)$, respectively.
\begin{figure}[ht]
    \centering

\tikzset{every picture/.style={line width=0.75pt}} %set default line width to 0.75pt        

\begin{tikzpicture}[x=0.75pt,y=0.75pt,yscale=-1,xscale=1]
%uncomment if require: \path (0,771); %set diagram left start at 0, and has height of 771
 
\draw    (180,300) -- (250,330) ;
%\draw    (180,360) -- (250,390) ;
\draw    (180,390) -- (250,300) ;

\draw    (380,330) -- (450,360) ;
\draw    (380,360) -- (450,300) ;
\draw    (380,390) -- (450,330) ;
%Straight Lines [id:da19211380368200826] 

% Text Node
\draw (160,290) node [anchor=north west][inner sep=0.75pt]  [font=\small]  {$1_1$};
\draw (160,320) node [anchor=north west][inner sep=0.75pt]  [font=\small]  {$2_1$};
\draw (160,350) node [anchor=north west][inner sep=0.75pt]  [font=\small]  {$3_1$};
\draw (160,380) node [anchor=north west][inner sep=0.75pt]  [font=\small]  {$4_1$};
\draw (255,290) node [anchor=north west][inner sep=0.75pt]  [font=\small]  {$1_2$};
\draw (255,320) node [anchor=north west][inner sep=0.75pt]  [font=\small]  {$2_2$};
\draw (255,350) node [anchor=north west][inner sep=0.75pt]  [font=\small]  {$3_2$};
\draw (255,380) node [anchor=north west][inner sep=0.75pt]  [font=\small]  {$4_2$};
\draw (210,400) node [anchor=north west][inner sep=0.75pt]  [font=\small]  {$H_1$};

\draw (360,290) node [anchor=north west][inner sep=0.75pt]  [font=\small]  {$1_3$};
\draw (360,320) node [anchor=north west][inner sep=0.75pt]  [font=\small]  {$2_3$};
\draw (360,350) node [anchor=north west][inner sep=0.75pt]  [font=\small]  {$3_3$};
\draw (360,380) node [anchor=north west][inner sep=0.75pt]  [font=\small]  {$4_3$};
\draw (455,290) node [anchor=north west][inner sep=0.75pt]  [font=\small]  {$1_4$};
\draw (455,320) node [anchor=north west][inner sep=0.75pt]  [font=\small]  {$2_4$};
\draw (455,350) node [anchor=north west][inner sep=0.75pt]  [font=\small]  {$3_4$};
\draw (455,380) node [anchor=north west][inner sep=0.75pt]  [font=\small]  {$4_4$};
\draw (410,400) node [anchor=north west][inner sep=0.75pt]  [font=\small]  {$H_2$};

% Text Node
%\draw (223.5,437.44) node [anchor=north west][inner sep=0.75pt]  [font=\small]  {$G$};

\filldraw[black] (180,300) circle (1.5pt) ;
\filldraw[black] (180,330) circle (1.5pt); 
\filldraw[black] (180,360) circle (1.5pt); 
\filldraw[black] (180,390) circle (1.5pt); 
\filldraw[black] (250,300) circle (1.5pt) ;
\filldraw[black] (250,330) circle (1.5pt); 
\filldraw[black] (250,360) circle (1.5pt); 
\filldraw[black] (250,390) circle (1.5pt); 

\filldraw[black] (380,300) circle (1.5pt) ;
\filldraw[black] (380,330) circle (1.5pt); 
\filldraw[black] (380,360) circle (1.5pt); 
\filldraw[black] (380,390) circle (1.5pt); 
\filldraw[black] (450,300) circle (1.5pt) ;
\filldraw[black] (450,330) circle (1.5pt); 
\filldraw[black] (450,360) circle (1.5pt); 
\filldraw[black] (450,390) circle (1.5pt); 
\end{tikzpicture}
\caption{The graphs $H_1$ and $H_2$}
\label{fig.hq}
\end{figure}
\end{eg}

\begin{remark}
    The motivation behind Equation (\ref{eq.matchcond}) is purely algebraic. Specifically, it aids in obtaining a monomial order on $S$ such that the leading terms of the elements in the minimal generating set of the twisted LSS-ideal are pairwise coprime.
\end{remark}

\begin{remark}
    The graph $H_q$ defined in Definition \ref{def.tmd} is constructed for computational purposes. More precisely, by using Remark \ref{pos.match}, one can conclude the existence of a positive map (twisted positive map) on the graph $H_q$ with respect to the given matching. For further details, see Example \ref{eg.posmap}.
\end{remark}

\begin{definition} \label{def.tpmd}
    Let $G$ be a graph. Given a graph $H_q$ with respect to a twisted matching decomposition $E(G)=\cup_{\ell=1}^{2p}M_\ell$ of $G$, a twisted positive mapping is a weight function  $w: V(H_q) \rightarrow \mathbb{R}$ satisfying: 
    \begin{equation*}
    \begin{split}
       w(i_{2q-1})+w(j_{2q})>0 \text{ and } w(j_{2q-1})+w(i_{2q})<0 & \text{ if } \{i_{2q-1},j_{2q}\} \in E(H_q), \\
       %&w(i_{2q-1})+w(j_{2q})<0 \text{ and } w(j_{2q-1})+w(i_{2q})>0 \text{ if } \{j_{2q-1},i_{2q}\} \in E(H_q), \\
       w(i_{2q-1})+w(j_{2q})<0 \text{ and } w(j_{2q-1})+w(i_{2q})<0 & \text{ if } \begin{cases}
           \{i,j\} \in E \setminus \cup_{k=0}^{2(q-1)}M_k, \text{ where $M_0=\emptyset$, and } \\
           \{i_{2q-1},j_{2q}\} \text{ and } \{j_{2q-1},i_{2q}\} \notin E(H_q), 
       \end{cases}
    \end{split}
    \end{equation*}  
    for all $i,j \in V$ $(i \neq j)$.
\end{definition}

\begin{eg} \label{eg.posmap}
    Let $G$ be a graph and $E=M_1 \cup \cdots \cup M_4$ be a twisted matching decomposition of $G$ as given in Example \ref{eg.hq}. In Figure \ref{fig.whq}, we define the twisted positive mappings $w_1$ and $w_2$, illustrated in red next to the vertices for the graphs $H_1$ and $H_2$, respectively. It is easy to verify that the weight functions $w_1$ and $w_2$ satisfies the conditions given in Definition \ref{def.tpmd}, where the sum of the weights of the vertices of each black edge is positive, and the sum of the weights of the vertices of each blue dotted edge is negative.

\begin{figure}[ht]
    \centering

\tikzset{every picture/.style={line width=0.75pt}} %set default line width to 0.75pt        

\begin{tikzpicture}[x=0.75pt,y=0.75pt,yscale=-1,xscale=1]
%uncomment if require: \path (0,771); %set diagram left start at 0, and has height of 771
 
\draw    (180,300) -- (250,330) ;
%\draw    (180,360) -- (250,390) ;
\draw    (180,390) -- (250,300) ;
\draw[dashed,blue]    (180,330) -- (250,300) ;
\draw[dashed,blue]    (180,300) -- (250,390) ;
\draw[dashed,blue]    (180,330) -- (250,360) ;
\draw[dashed,blue]    (180,360) -- (250,330) ;
\draw[dashed,blue]    (180,300) -- (250,360) ;
\draw[dashed,blue]    (180,360) -- (250,300) ;
\draw[dashed,blue]    (180,330) -- (250,390) ;
\draw[dashed,blue]    (180,390) -- (250,330) ;

\draw    (380,330) -- (450,360) ;
\draw    (380,360) -- (450,300) ;
\draw    (380,390) -- (450,330) ;
\draw[dashed,blue]    (380,360) -- (450,330) ;
\draw[dashed,blue]    (380,300) -- (450,360) ;
\draw[dashed,blue]    (380,330) -- (450,390) ;
%Straight Lines [id:da19211380368200826] 

% Text Node
\draw (160,290) node [anchor=north west][inner sep=0.75pt]  [font=\small]  {$1_1$};
\draw (160,320) node [anchor=north west][inner sep=0.75pt]  [font=\small]  {$2_1$};
\draw (160,350) node [anchor=north west][inner sep=0.75pt]  [font=\small]  {$3_1$};
\draw (160,380) node [anchor=north west][inner sep=0.75pt]  [font=\small]  {$4_1$};
\draw (255,290) node [anchor=north west][inner sep=0.75pt]  [font=\small]  {$1_2$};
\draw (255,320) node [anchor=north west][inner sep=0.75pt]  [font=\small]  {$2_2$};
\draw (255,350) node [anchor=north west][inner sep=0.75pt]  [font=\small]  {$3_2$};
\draw (255,380) node [anchor=north west][inner sep=0.75pt]  [font=\small]  {$4_2$};
\draw (210,400) node [anchor=north west][inner sep=0.75pt]  [font=\small]  {$H_1$};

%weight for h1
\draw (145,290) node [anchor=north west][inner sep=0.75pt]  [font=\small]  {\textcolor{red}{$2$}};
\draw (135,320) node [anchor=north west][inner sep=0.75pt]  [font=\small]  {\textcolor{red}{$-3$}};
\draw (135,350) node [anchor=north west][inner sep=0.75pt]  [font=\small]  {\textcolor{red}{$-3$}};
\draw (145,380) node [anchor=north west][inner sep=0.75pt]  [font=\small]  {\textcolor{red}{$2$}};
\draw (270,290) node [anchor=north west][inner sep=0.75pt]  [font=\small]  {\textcolor{red}{$-1$}};
\draw (270,320) node [anchor=north west][inner sep=0.75pt]  [font=\small]  {\textcolor{red}{$-1$}};
\draw (270,350) node [anchor=north west][inner sep=0.75pt]  [font=\small]  {\textcolor{red}{$-3$}};
\draw (270,380) node [anchor=north west][inner sep=0.75pt]  [font=\small]  {\textcolor{red}{$-3$}};

\draw (360,290) node [anchor=north west][inner sep=0.75pt]  [font=\small]  {$1_3$};
\draw (360,320) node [anchor=north west][inner sep=0.75pt]  [font=\small]  {$2_3$};
\draw (360,350) node [anchor=north west][inner sep=0.75pt]  [font=\small]  {$3_3$};
\draw (360,380) node [anchor=north west][inner sep=0.75pt]  [font=\small]  {$4_3$};
\draw (455,290) node [anchor=north west][inner sep=0.75pt]  [font=\small]  {$1_4$};
\draw (455,320) node [anchor=north west][inner sep=0.75pt]  [font=\small]  {$2_4$};
\draw (455,350) node [anchor=north west][inner sep=0.75pt]  [font=\small]  {$3_4$};
\draw (455,380) node [anchor=north west][inner sep=0.75pt]  [font=\small]  {$4_4$};
\draw (410,400) node [anchor=north west][inner sep=0.75pt]  [font=\small]  {$H_2$};

%weight for h2
\draw (335,290) node [anchor=north west][inner sep=0.75pt]  [font=\small]  {\textcolor{red}{$-3$}};
\draw (345,320) node [anchor=north west][inner sep=0.75pt]  [font=\small]  {\textcolor{red}{$2$}};
\draw (335,350) node [anchor=north west][inner sep=0.75pt]  [font=\small]  {\textcolor{red}{$-2$}};
\draw (345,380) node [anchor=north west][inner sep=0.75pt]  [font=\small]  {\textcolor{red}{$2$}};
\draw (480,290) node [anchor=north west][inner sep=0.75pt]  [font=\small]  {\textcolor{red}{$3$}};
\draw (470,320) node [anchor=north west][inner sep=0.75pt]  [font=\small]  {\textcolor{red}{$-1$}};
\draw (470,350) node [anchor=north west][inner sep=0.75pt]  [font=\small]  {\textcolor{red}{$-1$}};
\draw (470,380) node [anchor=north west][inner sep=0.75pt]  [font=\small]  {\textcolor{red}{$-3$}};
% Text Node
%\draw (223.5,437.44) node [anchor=north west][inner sep=0.75pt]  [font=\small]  {$G$};

\filldraw[black] (180,300) circle (1.5pt) ;
\filldraw[black] (180,330) circle (1.5pt); 
\filldraw[black] (180,360) circle (1.5pt); 
\filldraw[black] (180,390) circle (1.5pt); 
\filldraw[black] (250,300) circle (1.5pt) ;
\filldraw[black] (250,330) circle (1.5pt); 
\filldraw[black] (250,360) circle (1.5pt); 
\filldraw[black] (250,390) circle (1.5pt); 

\filldraw[black] (380,300) circle (1.5pt) ;
\filldraw[black] (380,330) circle (1.5pt); 
\filldraw[black] (380,360) circle (1.5pt); 
\filldraw[black] (380,390) circle (1.5pt); 
\filldraw[black] (450,300) circle (1.5pt) ;
\filldraw[black] (450,330) circle (1.5pt); 
\filldraw[black] (450,360) circle (1.5pt); 
\filldraw[black] (450,390) circle (1.5pt); 
\end{tikzpicture}
\caption{The weight functions for the graphs $H_1$ and $H_2$}
\label{fig.whq}
\end{figure}
\end{eg}

\begin{definition}
    We say a graph $G=(V(G),E(G))$ admits a twisted positive matching decomposition with respect to a twisted matching decomposition $E(G)=\cup_{i=1}^{2p}M_i$ if for $q=1, \ldots, p$, each graph $H_q$  has a twisted positive mapping. 
    %The pair $(M_{2q-1},M_{2q})$ with respect to a twisted matching decomposition of $E$ is called a part of the twisted matching decomposition. 
    The smallest $p$ for which $G$ admits a twisted positive matching decomposition with $p$ parts will be denoted by $\tpmd(G)$.
\end{definition}

\begin{eg}
    Let $m>1$ be an integer. Let $G=K_{1,m}$ be a star graph. Then $\tpmd(G)=\lceil m/2\rceil$.
    \begin{proof}
        Let $V(G)=\{0,1,\ldots,m\}$ be a vertex set and $E(G)=\cup_{i=1}^{m}M_i$, where $M_i=\{0,i\}$, be a twisted matching decomposition. If $m$ is odd then set $M_{m+1}=\{\emptyset\}$. Then, from Remark \ref{pos.match}, it follows that $H_q$ has twisted positive mapping for all $q=1,2,\ldots,\lceil m/2\rceil$, as desired.
    \end{proof}
%\item Let $G=C_m$ be a cycle then $\tpmd(G)=2$.
\end{eg}

\begin{prop} \label{compare}
    Let $G$ be a graph. Then $\lceil \Delta(G)/2 \rceil \leq \tpmd(G)\leq \pmd(G)$.
\end{prop}
\begin{proof}
    The inequality $\lceil \Delta(G)/2 \rceil \leq \tpmd(G)$ follows from matching decomposition of $G$. To prove the inequality $\tpmd(G)\leq \pmd(G)$, we assume that $\pmd(G)=p$ and $w_\ell:V(G) \rightarrow \mathbb{R}$ be a respective weight function on matching decomposition $E_\ell$ of $G$, for $\ell = 1,2,\ldots, p$. We need to show that $G$ admits a twisted positive matching decomposition with $p$ parts. Consider a twisted matching decomposition $(M_{2q-1},M_{2q})=(E_q,\{\emptyset \})$ for all $q = 1,2,\ldots, p$. Set $t_\ell=\max\{w_\ell(i)\mid i=1,2,\ldots, n\}+1$. We define a twisted mapping $w'_q:V(H_q) \rightarrow \mathbb{R}$ as follows:
    
    for $i<j \in V$ and $\{i,j\}\in (M_{2q-1},M_{2q})$,  
    $$w'_q(i_{2q-1})=w_q(i) \text{ and } w'_q(j_{2q})=w_q(j),$$
    $$w'_q(j_{2q-1})=-t_q \text{ and } w'_q(i_{2q})=-t_q,$$

    for all $i \in V\setminus V(E_q)$,
    $$w'_q(i_{2q-1})=-t_q \text{ and } w'_q(i_{2q})=-t_q.$$
    %$$w'_q(j_{2q-1})=-t_q \text{ and } w'_q(i_{2q})=-t_q.$$
    Since the weight functions $w_l$ for $l=1,2, \ldots, p$, correspond to the positive matching of $G$, it follows that $w'_q$ is a twisted positive mapping for all $q=1,2,\ldots, p$, as desired.
\end{proof}

\begin{remark}
It is important to note that in general, $\tpmd$ is not always equal to $\lceil \pmd/2 \rceil$. For instance, in case of a tree $G$, the corresponding $\pmd$ is given by $\Delta(G)$ but $\tpmd$ is not equal to $\lceil \Delta(G)/2 \rceil$ which is clear from the following example. 
Let $G=P_3$ be a path graph. Clearly, twisted matching decomposition of $G$ satisfying Equation (\ref{eq.matchcond}) have $3$ parts. This implies that $\tpmd(G)=2$ but $\lceil \Delta(G)/2 \rceil = 1$.  
\end{remark}

%Next, we prove the radical complete intersection property of the twisted LSS-ideals in Theorem \ref{CI_pmd} and give proof for Theorem \ref{CI_LSS}.  

For the following notations related to Gr\"obner basis theory, please refer \cite{Bruns&Conca}.

\begin{definition}
Let $S=k[x_1, \ldots,x_n]$ be a polynomial ring. For a non-zero polynomial $$f= \sum_{\alpha \in \NN^{n}}a_{\alpha}x^{\alpha}$$ and a vector $\mathfrak{w}=(w_i: \, i \in [n]) \in \mathbb{R}^{n}$, set $m_{\mathfrak{w}}(f)=\max_{a_{\alpha} \neq 0} \{ \alpha \cdot \mathfrak{w} \}$. Then, $$ \text{in}_{\mathfrak{w}}(f)= \sum_{\alpha \cdot \mathfrak{w}=m_{\mathfrak{w}}(f)}a_{\alpha}x^{\alpha} $$ is called the initial form of $f$ with respect to $\mathfrak{w}.$
\end{definition}

Let $S$ be a polynomial ring and let $I$ be an ideal in $S$. Then for a term order $<$ on $S$ and $f \in S$, in$_<f$ denotes the largest term of $f$ and in$_<I$ denotes the ideal generated by in$_<f$ with $f \in I \backslash \{0\}$.

\begin{prop} \label{LT}
     Let $G=(V(G),E(G))$ be a graph, $d \geq p=\tpmd(G)$ and $E(G)= \cup_{\ell=1}^{2p}M_{\ell}$ be a twisted matching decomposition. If $G$ admits a twisted positive matching decomposition with respect to twisted matching decomposition $E(G)=\cup_{\ell=1}^{2p}M_\ell$, then there exists a term order $<$ on $S$ such that for all  $q=1,2,\ldots,p$, for every $A=\{ i,j \} \in (M_{2q-1},M_{2q})$,
    \begin{equation} \label{LT-eq}
     \text{in}_<(\hat{f}_A^{(d)})= \begin{cases}
         x_{i2q-1}x_{j2q} \hspace{5mm} \text{ if } \{i,j\} \in M_{2q-1}\\
         x_{i2q}x_{j2q-1} \hspace{5mm} \text{ if } \{ i,j \} \in M_{2q}.
     \end{cases}
    \end{equation}
\end{prop}
\begin{proof}
Consider $E(G)= \cup_{\ell=1}^{2p}M_{\ell}$ be the matching decomposition of $G$ and $E(G) = \cup_{q=1}^p(M_{2q-1},M_{2q})$ be a twisted matching decomposition with the respective weight functions $w_{q}:V(H_q) \rightarrow \mathbb{R}$. In order to define the required term order $<$, we define the weight vectors $\mathfrak{w}_1, \ldots, \mathfrak{w}_p \in \mathbb{R}^{|V| \times 2d}$ as follows,

\begin{itemize}
    \item $\mathfrak{w}_{q}(x_{ik})=0$ if $k \notin \{2q-1,2q\}$ and
    \item $\mathfrak{w}_{q}(x_{ik})=w_q(i_k)$ if $k \in \{2q-1,2q\}$.
\end{itemize}

Then, by construction, it follows that:
\begin{equation} \label{intial1}
    \text{in}_{\mathfrak{w}_{1}}(\hat{f}_{\{i,j\}}^{(d)})= \begin{cases}
    x_{i1}x_{j2} \hspace{5mm} \text{ if } \{i,j\} \in M_{1}, \\
    x_{i2}x_{j1} \hspace{5mm} \text{ if } \{i,j\} \in M_{2}, \\
    \sum_{k=2}^d(x_{i2k-1}x_{j2k}-x_{i2k}x_{j2k-1}) \hspace{5mm} \text{ if } \{i,j\} \notin (M_{1},M_2).
\end{cases}
\end{equation}

With the weight vectors defined, we say $x^{\alpha} < x^{\beta}$ if,
\begin{enumerate}
    \item $|\alpha|<|\beta|$ or 
    \item $|\alpha|=|\beta|$ and $\mathfrak{w}_{q}(x^{\alpha})<\mathfrak{w}_{q}(x^{\beta})$ for the smallest $q$ such that $\mathfrak{w}_{q}(x^{\alpha})\neq \mathfrak{w}_{q}(x^{\beta}) $ or
    \item $|\alpha|=|\beta|$ and $\mathfrak{w}_{q}(x^{\alpha})=\mathfrak{w}_{q}(x^{\beta})$ for all $q$ and $x^{\alpha}<_0x^{\beta}$ for an arbitrary but fixed term order $<_0$.
\end{enumerate}

For a given edge $A=\{i,j\} \in E(G)$, one has $A \in (M_{2q-1},M_{2q})$ for some $1\leq q \leq p$. There are two possibilities, either $A \in (M_{1},M_{2})$ or $A \notin (M_{1},M_{2})$.

If $A \in (M_{1},M_{2})$, then the statement follows from Equation (\ref{intial1}).

If $A \notin (M_{1},M_{2})$, then, from the defined weight vector $\mathfrak{w}_{1}$, it follows that $\mathfrak{w}_{1}(x_{i1}x_{j2})<\mathfrak{w}_{1}(y)$ and $\mathfrak{w}_{1}(x_{i2}x_{j1})<\mathfrak{w}_{1}(y)$ for all $y \in \{x_{i,2k-1}x_{j,2k},x_{i,2k}x_{j,2k-1}\}$, where $k=2,\ldots, d$. This implies that $x_{i1}x_{j2}$ and $x_{i2}x_{j1}$ are not the leading term of $\hat{f}_A^{(d)}$ with respect to $<$, and so the leading term is a monomial in $\sum_{k=2}^d(x_{i2k-1}x_{j2k}-x_{i2k}x_{j2k-1})$. Now, in the next step, there are two possibilities, either $A \in (M_{3},M_{4})$ or $A \notin (M_{3},M_{4})$. 

If $A \in (M_{3},M_{4})$, then, from the weight vector $\mathfrak{w}_{2}$, it follows that $x_{i3}x_{j4} (x_{i4}x_{j3})$ is the leading term of $\hat{f}_A^{(d)}$ if $A \in M_3 (M_4)$, respectively. 

If $A \notin (M_{3},M_{4})$, then, from the weight vector $\mathfrak{w}_{2}$, it follows that $\mathfrak{w}_{2}(x_{i3}x_{j4})<\mathfrak{w}_{2}(y)$ and $\mathfrak{w}_{2}(x_{i4}x_{j3})<\mathfrak{w}_{2}(y)$ for all $y \in \{x_{i,2k-1}x_{j,2k},x_{i,2k}x_{j,2k-1}\}$, where $k=3,\ldots, d$. This implies that with respect to the monomial order $<$, $x_{i3}x_{j4}$ and $x_{i4}x_{j3}$ are not the leading term of $\hat{f}_A^{(d)}$, and so the leading term is a monomial in $\sum_{k=3}^d(x_{i2k-1}x_{j2k}-x_{i2k}x_{j2k-1})$. 
In this way, by repeating the above process one obtains the desired result. 

Observe that this process terminates in a finite number of steps as the leading term of $\hat{f}_A^{(d)}$ will be obtained at the $q$th step, where $q \leq p$ is the integer such that $A \in (M_{2q-1},M_{2q})$.
\end{proof}

\begin{eg}
 Let $G$ be a graph and $E(G)=M_1 \cup \cdots \cup M_4$ be a twisted matching decomposition of $G$ as given in Example \ref{eg.hq}. Let $d \geq 2$ be an integer and $S=K[x_{ij} \mid i \in [4], j \in [2d]]$ be a polynomial ring. The twisted LSS-ideal of $G$ generated by the following elements:
    
    $\hat{f}_{12}^{(d)}= \textcolor{red}{x_{11}x_{22}}-x_{12}x_{21}+x_{13}x_{24}-x_{14}x_{23}+\cdots + x_{1,2d-1}x_{2,2d}-x_{1,2d}x_{2,2d-1}$,
    
    $\hat{f}_{13}^{(d)}= x_{11}x_{32}-x_{12}x_{31}+x_{13}x_{34}-\textcolor{red}{x_{14}x_{33}}+\cdots + x_{1,2d-1}x_{3,2d}-x_{1,2d}x_{3,2d-1}$,

    $\hat{f}_{14}^{(d)}= x_{11}x_{42}-\textcolor{red}{x_{12}x_{41}}+x_{13}x_{44}-x_{14}x_{43}+\cdots + x_{1,2d-1}x_{4,2d}-x_{1,2d}x_{4,2d-1}$,
    
    $\hat{f}_{23}^{(d)}= x_{21}x_{32}-x_{22}x_{31}+\textcolor{red}{x_{23}x_{34}}-x_{24}x_{33}+\cdots + x_{2,2d-1}x_{3,2d}-x_{2,2d}x_{3,2d-1}$,
    
    $\hat{f}_{24}^{(d)}= {x_{21}x_{42}}-x_{22}x_{41}+x_{23}x_{44}-\textcolor{red}{x_{24}x_{43}}+\cdots + x_{2,2d-1}x_{4,2d}-x_{2,2d}x_{4,2d-1}$.
    
    \vspace{2mm}

\noindent In the above, we highlighted leading term of $\hat{f}_{e}^{(d)}$ in red, for all $e\in E(G)$, with respect to term order $<$ as defined in Proposition \ref{LT} on $S$. Observe that leading term of $\hat{f}_{e}^{(d)}$ for all $e \in E(G)$ is pairwise coprime for $d \geq 2$. Hence, the ideal $L_G(d)$ is a complete intersection for $d \geq 2$. Furthermore, this twisted matching decomposition is minimal since $\Delta(G)=3$, this implies that $\tpmd(G) = 2$.
\end{eg}

\begin{remark}
The weight vectors play a crucial role in forcing the leading terms of the generators of the ideal to be coprime to each other, which eventually helps in concluding the complete intersection property of the ideal. 
Note that the weight vectors defined by the authors in \cite[Lemma $5.5$]{CW2019} cannot be used directly to get relatively prime leading terms 
in the case of twisted LSS-ideals because of the change in the form of the generators. In fact, we observe that a slight change in defining the weight functions and the weight vectors seems to work in the case of twisted LSS-ideals but the bound obtained is far from being optimal.
Hence, in order to obtain the desired leading terms with a better bound on $d$, in Proposition \ref{LT} we consider the weight vectors corresponding to the twisted positive matching decomposition. 
\end{remark}

As a consequence of Proposition \ref{LT}, we obtain the radical complete intersection property of the twisted LSS-ideals. 

\begin{thm} \label{CI_pmd}
    Let $G$ be a graph. Then $\hat{L}_G(d)$ is a radical complete intersection when $d \geq \tpmd (G)$.
\end{thm}
\begin{proof}
    Let $d \geq p= \tpmd(G)$ and $E=\cup_{q=1}^p(M_{2q-1},M_{2q})$ be a twisted matching decomposition of $G$. Since Proposition $\ref{LT}$ guarantees a term order $<$ satisfying Equation $\ref{LT-eq}$ and $E$ is a twisted matching decomposition of $G$, one obtains that the initial monomials of $\hat{f}_{A}^{(d)}$ of $\hat{L}_G(d)$ are pairwise coprime and squarefree. Then, the result follows from Remark \ref{rem.reg.seq}. 
\end{proof}

The following result gives the necessary conditions %(Lemma \ref{notCI}) 
for LSS-ideals and twisted LSS-ideals of graphs, in general, to be complete intersections.

\begin{lemma} \label{notCI}
    Let $d$ be an integer and $G$ be a graph. Then,
    \begin{enumerate}
        \item \label{lem.CI1} if $d< \Delta(G)$ then $L_{G}(d)$ is not a complete intersection.
        \item \label{lem.CI2} if $2d< \Delta(G)$ then $\hat{L}_{G}(d)$ is not a complete intersection.
    \end{enumerate}
\end{lemma}
\begin{proof}
    $(\ref{lem.CI1})$. It suffices to show that there exists a prime $P$ such that $L_G(d)$ is contained in $P$ with $\het(P)\leq \mu(L_G(d))-1$, where $\mu(L_G(d))$ denotes the cardinality of a minimal generating set of $L_G(d)$. To prove this, let $u$ be a vertex of $G$ such that $\deg_G(u) \geq d+1$ and set $T=\{u\}$. Assume $P_T=(x_{u1},\ldots, x_{ud})+ Q_{G\setminus u}$, where $Q_{G\setminus u}$ is a minimal prime of $L_{G\setminus u}(d)$. Since the prime ideals $Q_{G\setminus u}(d)$ and $(x_{u1},\ldots, x_{ud})$ are in distinct set of variables, $P_T$ is a prime ideal and clearly contains $L_G(d)$. Now, from \cite[Theorem 13.5]{M1989}, one has, $\het(Q_{G\setminus u}) \leq \mu(L_{G \setminus u}(d))$. Then,  
    \begin{equation*}
    \begin{split}
         \het(P_T) & = \het(x_{u1},\ldots, x_{ud}) + \het(Q_{G\setminus u}), \\
        & \leq d + \mu(L_{G \setminus u}(d)), \\
        & = d+ \mu(L_{G}(d)) - (d+1), \\
        & = \mu(L_{G}(d))-1.
    \end{split}
    \end{equation*}
    Thus $L_G(d)$ is not a complete intersection.

The assertion $(\ref{lem.CI2})$ follows in a similar way.
\end{proof}

In Theorem \ref{Thm:UCI}, we give sufficient conditions for the radical complete intersection property of the LSS-ideals corresponding to unicyclic and bicyclic graphs in terms of the maximum degree of vertices of the graphs. 

\begin{thm} \label{Thm:UCI}
    Let $G$ be a graph and $d$ a positive integer. 
\begin{enumerate}
    \item If $G$ is a unicyclic graph with $d \geq 3$, then $L_{G}(d)$ is a radical complete intersection for $d \geq \Delta(G)$. 
    \item If $G$ is a bicyclic graph with $d \geq 4$, then  $L_{G}(d)$ is a radical complete intersection for $d \geq \Delta(G)$. 
\end{enumerate}
  \end{thm}
\begin{proof}

\noindent (1). By Remark \ref{pmdCI}, it is enough to show $\pmd(G) \leq d$. We consider a matching $M_1$ consisting of an edge $e_1$ of the cycle and edges $e_2, \ldots, e_m$ such that for $i=2, \ldots, m$, $e_i$ is not an edge of the cycle and have a vertex of degree $\Delta(G)$ in $G \setminus \{e_1, \ldots, e_{i-1}\}$. Since $M_1$ has only one edge from the cycle, from Remark \ref{pos.match}, $M_1$ is a positive matching on $G$. From the construction of $M_1$, we get that  $G\setminus M_1$ is a forest with $\Delta(G\setminus M_1) = \max\{2, \Delta(G)-1\}$. Since $G\setminus M_1$ is a forest it follows from Remark \ref{rem.pmdforest}, it follows that $\pmd(G\setminus M_1) = \max\{2, \Delta(G)-1\}$. This implies that $\pmd(G) \leq \max\{3,\Delta(G)\}$, in particular, $\pmd(G) \leq d$. 

\vspace{2mm}

\noindent (2). The proof is similar to part $(1)$ of this theorem. We consider a positive matching $M_1$ consisting of an edge $e_1$ of a cycle and edges $e_2, \ldots, e_m$ such that for $i=2, \ldots, m$, $e_i$ is not an edge of any cycle and have a vertex of degree $\Delta(G)$ in $G \setminus \{e_1, \ldots, e_{i-1}\}$. Then, in this case, from the construction of the matching $M_1$, $G\setminus M_1$ is unicyclic with $\Delta(G\setminus M_1) = \max\{3,\Delta(G)-1\}$. Moreover, from $(1)$,  $\pmd(G\setminus M_1) \leq \max\{3,\Delta(G)-1\}$. Hence we get, $\pmd(G) \leq \max\{4,\Delta(G)\}$ and so the theorem.
\end{proof}

\begin{remark}
    Observe that as a result of Theorem \ref{Thm:UCI} and \cite[Theorem $1.1$]{CW2019}, one obtains conditions which guarantee the primality of the LSS-ideals corresponding to unicyclic (bicyclic) graphs.   
\end{remark}

As a consequence of the above theorems and \cite[Proposition 7.4, 7.5, 7.7]{CW2019}, sufficient conditions are obtained for the ideal of $(d+1)$-minors of generic/generic symmetric matrices associated with 
%unicyclic (bicyclic) 
graphs to be radical and of maximal height and the Pfaffian ideal generated by Pfaffians of order $2d+2$ of generic skew-symmetric matrices associated with graphs to be radical.

\begin{cor}
    Let $G$ be a unicyclic (bicyclic) graph with $d \geq \Delta(G)$ and $d \geq 3$. Then: 
    \begin{enumerate}
        \item 
        %$I_{d+1}(X_G^{gen})$, 
        $I_{d+1}(X_G^{sym})$ is radical and attains maximal height.
        \item If $G$ has a unique cycle of even length, then $I_{d+1}(X_G^{gen})$ is radical and attains maximal height. 
    \end{enumerate}
\end{cor}

\begin{cor}
Let $G$ be a unicyclic (bicyclic) graph with $2d\geq \Delta(G)$ and $d \geq 2$. Then Pf$_{2d+2}(X_G^{skew})$ is radical and attains maximal height.    
\end{cor}

\begin{remark}
    Note that Theorem \ref{Thm:UCI} fails to be true for $d=2$ and $d=3$, respectively. For example:  
    \begin{enumerate}
        \item  The ideal $L_{C_4}(2)$ is not a complete intersection by \cite[Theorem 3.5]{AK2021}. 
        \item  Let $G=K_{2,3}$ be a graph. Using Macaulay$2$ one can see that $\mu(L_G(3)) > \het(L_G(3))=5$, thus $L_G(3)$ is not a complete intersection.
    \end{enumerate}
\end{remark}

\subsection*{Conclusion} The proof of Theorem ~\ref{CI_LSS} follows from Lemma \ref{notCI} and Theorem \ref{Thm:UCI}.

\section{Almost Complete Intersection}

This section includes the necessary and sufficient conditions for LSS-ideals corresponding to trees and  $C_3$-free unicyclic (bicyclic) graphs to be almost complete intersections. We begin by giving a necessary condition for an LSS-ideal associated with a graph, in general, to be an almost complete intersection. 

\begin{lemma} \label{lemma:notCIACI}
    Let $G$ be a graph and $d$ be an integer such that $d < \Delta(G)-1$. Then $L_{G}(d)$ is not an almost complete intersection.
\end{lemma}
\begin{proof}
    Let $u$ be a vertex of $G$ with $\deg_G(u) \geq d+2$ %We claim there exists a prime $P$ such that $L_G(d)$ is contained in $P$ with $\het(P)\leq \mu(G)-2$. 
    and set $T=\{u\}$. Then the rest of the proof is similar to that of Lemma \ref{notCI}.
\end{proof}

\begin{thm} \label{ACI:tree}
    Let $d$ be a positive integer. If $G$ is a tree on $[n]$ with $\Delta(G)>d$, then $L_G(d)$ is an almost complete intersection if and only if $G$ is obtained by adding an edge between two trees $H_1$ and $H_2$ with $V(H_i)=V(G)$ and $\Delta(H_i)\leq d$ for $i=1,2$.
\end{thm}
%\begin{thm} \label{ACI:tree}
 %   Let $G$ be a tree on $[n]$ such that $\Delta(G)>d$ and $H_1$ and $H_2$ be trees with $\Delta(H_1)\leq d$ and $\Delta(H_2)\leq d$. Then $L_G(d)$ is an almost complete intersection if and only if $G$ is obtained by adding an edge between two vertices of two trees $H_1$ and $H_2$.
%\end{thm}
\begin{proof}
    %Let $H_1$ and $H_2$ be two trees such that $L_{H_1}(d)$ and $L_{H_2}(d)$ are complete intersection. 
    Suppose $G$ is obtained by adding an edge $e=\{u,v\}$ between $H_1$ and $H_2$, where $\Delta(H_1)\leq d$ and $\Delta(H_2)\leq d$. Then, $\deg_G(u)=d+1 \text{ or } \deg_G(v) = d+1$. From Remark \ref{R,CI,Prime}, it follows that $L_{G\setminus e}(d)$ is a radical complete intersection. Since $\Delta(G)>d$ from Lemma \ref{notCI}, Remark \ref{colonp} and Remark \ref{CIorACI}, it follows that $L_G(d)$ is an almost complete intersection.

    Now, assume that $G$ is not a graph obtained by adding an edge between $H_1$ and $H_2$. Then, either there exists a vertex $u$ such that $\deg_G(u) \geq d+2$ or there exist $v, w \in V(G)$ such that $ \deg_G(v)=d+1, \deg_G(w) = d+1$ and $\{v, w\} \notin E(G)$. We claim there exists a prime ideal $P \supseteq L_G(d)$ such that $\het(P)\leq n-3$. From this, it will follow that $\het(L_G(d)) \leq n-3$ and so $L_G(d)$ is not an almost complete intersection. 

\vspace{1mm}

    \textbf{Case I:} If there exists a vertex $u$ such that $\deg_G(u) \geq d+2$, then the claim follows from Lemma \ref{lemma:notCIACI}.

    \textbf{Case II:} If there exist $v, w \in V(G)$ such that $ \deg_G(v)=d+1, \deg_G(w) = d+1$ and $\{v, w\} \notin E(G)$, then set $T=\{v,w\}$. 
    \begin{equation} \label{eq:d}
    \begin{split}
        & \text{Input: } T \\
        & \text{WHILE }(G\setminus T \text{ has an vertex } u  \text{ such that } \deg_{G \setminus T}(u) \geq d) \\ 
        & \{ \\       
        & T= T \cup \{u\} \\     
        & \} \\
        & \text{RETURN } T
    \end{split}
\end{equation}

     %Then $T \subset V(G)$ and let $T=\{v,w, v_1, \ldots, v_m\}$. 
     From the construction of $T$, we have $\Delta(G\setminus T) \leq d-1$. Then from Remark \ref{R,CI,Prime}(c), $L_{G\setminus T}(d)$ is a prime ideal. We name the elements of $T$ as $v,w, v_1, \ldots, v_m$ and to this set, we associate an ideal $P_T$ given by 
     $$P_T=(x_{v1},\ldots, x_{vd},x_{w1},\ldots, x_{wd}, x_{v_11},\ldots x_{v_md})+L_{G\setminus T}(d).$$ Clearly, $P_T$ is a prime ideal containing $L_G(d)$. 
    
    Next, we compute the height of $P_T$. 
    %For which we first compute the height of $L_{G\setminus T}(d)$. 
    From Remark \ref{R,CI,Prime}, $G \setminus T$ is a complete intersection. Hence,  
    $$\het(L_{G\setminus T}(d))=\mu(L_{G \setminus u}(d)) = |E(G)| - \deg_G(v)- \deg_{G\setminus v}(w) -\sum_{i=0}^{m-1}\deg_{G\setminus \{v,w,v_1,\ldots,v_i\}}(v_{i+1}).$$
    Then, 
    \begin{equation} \label{ACI_ht}
    \begin{split}
        \het(P_T) & = \het(x_{v1},\ldots, x_{vd},x_{w1},\ldots, x_{wd}, x_{v_11},\ldots x_{v_md}) + \het(L_{G\setminus T}(d)), \\
        & = d(m+2) + n-1 - \deg_G(v)-\deg_{G\setminus v}(w)-\sum_{i=0}^{m-1}\deg_{G\setminus \{v,w,v_1,\ldots,v_i\}}(v_{i+1}).
    \end{split} 
    \end{equation}
    By the construction of $T$, we have $\deg_{G}(v) = d+1, \deg_{G\setminus v}(w) = d+1$, and  $\deg_{G\setminus \{v,w,v_1,\ldots,v_i\}}(v_{i+1}) \geq d$, for $i={0,\ldots,m-1}$. Substituting these values in Equation (\ref{ACI_ht}), we get, $\het(P_T) \leq n-3$, as desired.
\end{proof}

Next, we move on to look at the almost complete intersection LSS-ideals coming corresponding to unicyclic (bicyclic) graphs.

\begin{thm} \label{ACI:unicyclic}
    Let $d$ be a positive integer. Let $G$ be a connected $C_3$-free unicyclic graph on $[n]$ with $\Delta(G)>d$ and $d \geq 3$. Then $L_G(d)$ is an almost complete intersection if and only if $G$ has one of the following forms:
    \begin{enumerate}
        \item \label{ACIa} $G$ is obtained by adding an edge between vertices of a tree $H$ with $V(H)=V(G)$ and $\Delta(H) \leq d$;
        \item \label{ACIb} $G$ is obtained by adding an edge between a tree $H$ and a unicyclic graph $U$ with $V(H)=V(G)$ and $\Delta(H) \leq d$, and $V(U)=V(G)$ and $\Delta(U) \leq d$.
        %\item \label{ACIc} $G$ is obtained by attaching pendant vertices of $d-1$ trees $H_1,\ldots,H_{d-1}$ with $\Delta(H_i)\leq d$, where $i=1,\ldots,d-1$ to each vertex of $C_3$.
    \end{enumerate}
\end{thm}
\begin{proof}
    Assume that $L_G(d)$ is an almost complete intersection. Then $\het(L_G(d))=\mu(L_G(d))-1 = n-1$. From Lemma \ref{lemma:notCIACI}, it follows that $G$ does not have a vertex with $\deg_G(u) \geq d+2$. Now, we claim that if $G$ has two distinct vertices $u, v \in V(G)$ such that $\deg_G(u) = d+1$ and $\deg_G(v) = d+1$, then $\{u,v\} \in E(G)$. Suppose $\{u, v\} \notin E(G)$. Then, setting $T=\{u,v\}$ and proceeding along the same lines as the proof of Lemma \ref{notCI}, one gets $\het(L_G(d)) \leq n-2$. This contradicts the fact that $\het(L_G(d))= n-1$. Therefore, we get $\{u,v\} \in E(G)$. This implies if $G$ has three vertices of degree $d+1$, then $G$ has $C_3$ as an induced subgraph. Thus, the number of vertices of degree $d+1$ is at most $2$, since $G$ is a $C_3$-free unicyclic graph.  Hence, $G$ is either of type-(\ref{ACIa}) or type-(\ref{ACIb}).

\vspace{2mm}

    Conversely, suppose $G$ is of type-(\ref{ACIa}) or type-(\ref{ACIb}). Then there exists an edge $e\in E(G)$ such that $L_G(d) = L_{G\setminus e}+f_e^{(d)}$ and $\Delta(G\setminus e) = d$. Since $G\setminus e$ is either a tree or a unicyclic graph, Remark \ref{R,CI,Prime} and Theorem \ref{Thm:UCI}$(1)$ implies $L_{G\setminus e}(d)$ is a radical complete intersection. Therefore, from Remark \ref{colonp}, Remark \ref{CIorACI} and Lemma \ref{notCI}, it follows that $L_{G}(d)$ is an almost complete intersection.
\end{proof}

\begin{cor}
   Let $G = G_1\cup \cdots \cup G_m$ be a union of disconnected unicyclic graphs. Then $L_G(d)$ is an almost complete intersection if and only if for some $i$, $L_{G_i}(d)$ is an almost complete intersection and for $j \neq i$, $L_{G_j}(d)$ are complete intersections.
\end{cor}

\vspace{1mm}

\begin{remark}
Let $G$ be a graph on $[9]$. 

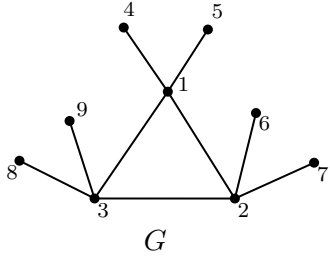
\begin{figure}[ht]
    \centering

\tikzset{every picture/.style={line width=0.75pt}}  

\begin{tikzpicture}[x=0.75pt,y=0.75pt,yscale=-1,xscale=1]

\draw   (443.54,172) -- (476.99,225.75) -- (406.99,225.72) -- cycle ;

\draw    (421.53,139.73) -- (443.54,172) ;

\draw    (463.53,140.73) -- (443.54,172) ;

\draw    (476.99,225.75) -- (516.53,207.73) ;

\draw    (476.99,225.75) -- (487.53,182.73) ;

\draw    (394.53,186.73) -- (406.99,225.72) ;

\draw    (369.53,206.73) -- (406.99,225.72) ;

\draw (447,163.5) node [anchor=north west][inner sep=0.75pt]  [font=\scriptsize]  {$1$};

\draw (476.99,226.75) node [anchor=north west][inner sep=0.75pt]  [font=\scriptsize]  {$2$};

\draw (406.99,226.72) node [anchor=north west][inner sep=0.75pt]  [font=\scriptsize]  {$3$};

\draw (420,125.5) node [anchor=north west][inner sep=0.75pt]  [font=\scriptsize]  {$4$};

\draw (464,126.5) node [anchor=north west][inner sep=0.75pt]  [font=\scriptsize]  {$5$};

\draw (487.53,182.73) node [anchor=north west][inner sep=0.75pt]  [font=\scriptsize]  {$6$};

\draw (516.53,208.73) node [anchor=north west][inner sep=0.75pt]  [font=\scriptsize]  {$7$};

\draw (396.53,175.73) node [anchor=north west][inner sep=0.75pt]  [font=\scriptsize]  {$9$};

\draw (361.53,207.73) node [anchor=north west][inner sep=0.75pt]  [font=\scriptsize]  {$8$};

\draw (430,240) node [anchor=north west][inner sep=0.75pt]  [font=\normalsize]  {$G$};

\filldraw[black] (421.53,139.73) circle (1.5pt) ;
\filldraw[black] (443.54,172) circle (1.5pt) ;
\filldraw[black] (463.53,140.73) circle (1.5pt) ;
\filldraw[black] (516.53,207.73) circle (1.5pt) ;
\filldraw[black] (476.99,225.75)  circle (1.5pt) ;
\filldraw[black] (487.53,182.73) circle (1.5pt) ;
\filldraw[black] (394.53,186.73) circle (1.5pt) ;
\filldraw[black] (406.99,225.72) circle (1.5pt) ;
\filldraw[black] (369.53,206.73)  circle (1.5pt) ;
\end{tikzpicture}
\end{figure}
From Macaulay$2$ computations, we get $L_G(3)$ to be an almost complete intersection. In fact, looking at the graphs of similar form (containing $C_3$), we observe a class of almost complete intersection LSS-ideals being associated with it. Therefore, with ample computational evidence, we ask the following question. 
%That is $\het(L_G(3))= \mu(L_G(3))-1=8$, and the ideal $L_G(d)_p$ is a complete intersection for all $p \in \Min(L_G(3))$.
\end{remark}

\begin{question} \label{triACI}
    Let $G$ be a connected unicyclic graph on $[n]$ and $d\geq 2$. If $G$ is obtained by attaching pendant vertices of $d-1$ trees $H_1,\ldots,H_{d-1}$ with $\Delta(H_i)\leq d$, where $i=1,\ldots,d-1$ to each vertex of $C_3$, is $L_G(d)$ an almost complete intersection?
\end{question}

If true, this, along with Theorem \ref{ACI:unicyclic} (dropping the assumption of $C_3$-free), will characterize unicyclic graphs whose associated LSS-ideals are almost complete intersections.

\begin{thm} \label{ACI:bicyclic}
    Let $d$ be a positive integer. Let $G$ be a connected $C_3$-free  bicyclic graph on $[n]$ with $\Delta(G)>d$ and $d \geq 4$. Then $L_G(d)$ is an almost complete intersection if and only if $G$ has one of the following forms:
    \begin{enumerate}
        \item \label{ACIbia} $G$ is obtained by adding an edge between vertices of a unicyclic graph $U$ with $V(U)=V(G)$ and $\Delta(U) \leq d$;
        \item \label{ACIbib} $G$ is obtained by adding an edge between unicyclic graphs $U_1$ and $U_2$ with $V(U_i)=V(G)$ and $\Delta(U_i) \leq d$ for $i=1,2$;
        \item \label{ACIbic} $G$ is obtained by adding an edge between a tree $H$ and a bicyclic graph $B$ with $V(H)=V(G)$ and $\Delta(H) \leq d$ and $V(B)=V(G)$ and $\Delta(B) \leq d$.
        
    \end{enumerate}
\end{thm}
\begin{proof}
    The proof is similar to that of Theorem \ref{ACI:unicyclic}.
\end{proof}

The following proposition is a consequence of isomorphisms mentioned in Remark \ref{Isomorphisms}(\ref{iso1}) and (\ref{iso2}).

\begin{prop} Let $d$ be an integer.
\begin{enumerate}
    \item For a subgraph $G$ of $K_{m,n}$ where $m,n \in \NN$, if $L_G(d)$ is an almost complete intersection, then height of $I_{d+1}(X_G^{gen})$ is one less than the maximal height.
    
    \item Let $G$ be a subgraph of $K_{n}$, where $n \in \NN $. If $L_G(d)$ is an almost complete intersection, then height of $I_{d+1}(X_G^{sym})$ is one less than the maximal height.
\end{enumerate}
\end{prop}

%\begin{proof}
%Let $G$ be either a subgraph of $K_{m,n}$ or that of $K_n$ such that $L_G(d)$ is an almost complete intersection. Then the result follows from the isomorphisms \ref{iso1} and \ref{iso2} mentioned in the first section.
%\end{proof}

%\begin{remark}
%    In particular, this implies that for all graphs of the form mentioned in Theorem \ref{ACI:tree}, Theorem \ref{ACI:unicyclic}, and Theorem \ref{ACI:bicyclic}, the corresponding ideal of $d+1$-minors of the associated generic matrices attain one less than the maximal height. 
%\end{remark}

From the results in this section, it thus follows that, for all graphs of the form mentioned in Theorem \ref{ACI:tree}, Theorem \ref{ACI:unicyclic}, and Theorem \ref{ACI:bicyclic}, the corresponding ideal of $d+1$-minors of the associated generic matrices attain one less than the maximal height.

%\textcolor{red}{As an immediate consequence, we have the following.}
%\begin{cor}
%Let $G$ be a graph on $[n]$.
%\begin{enumerate}
 %   \item If $G$ is a tree satisfying the conditions in Theorem \ref{ACI:tree} then,
  %  \begin{enumerate}
   %     \item for $n$ even, $\het(I_{d+1}(X_G^{gen}))=\frac{n^2}{4}-n(d-1)-2$.
    %    \item for $n$ odd, $\het(I_{d+1}(X_G^{gen}))=\frac{n^2}{4}-n(d-1)-\frac{9}{4}$.
%    \end{enumerate}    
%    \item If $G$ is a unicyclic graph on $[n]$ satisfying any one of the conditions in Theorem \ref{ACI:unicyclic} then, $\het(I_{d+1}(X_G^{sym}))=\frac{n^2}{2}-n(d-\frac{1}{2})-1$.
%\end{enumerate}
%\end{cor}

\section{Regularity}

Let $G$ be a graph and $d$ be a positive integer. In this section, we first compute the regularity of powers of LSS-ideals corresponding to trees and unicyclic graphs with $ \Delta(G) \leq d$. This is followed by associating certain invariants of $G$ in terms of the cardinality of the edges of its induced subgraphs and giving lower bounds for the regularity of powers of the related LSS-ideals.
Further, bounds are given for the regularity of powers of almost complete intersection LSS-ideals. %The section ends with looking at the Koszulness of quotients of LSS-ideals.

\vspace{1mm}

\begin{prop}\label{rem:CIreg}
    Let $G$ be a graph on $[n]$ with $ \Delta(G) \leq d$ and $d \geq 3$. 
    \begin{enumerate}
        \item If $G$ is a tree then for all $s\geq 1$, we have $\reg (S/L_{G}(d)^s) = 2s + n-3$.
        \item If $G$ is a connected unicyclic graph then we have $\reg (S/L_{G}(d)^s) = 2s + n-2$, for all $s\geq 1$.
    \end{enumerate}
\end{prop} 

\begin{proof}
    From Remark \ref{R,CI,Prime} and Theorem \ref{Thm:UCI}$(1)$, it follows that $L_G(d)$ is a complete intersection. Then the statement is a consequence of Remark \ref{regular}.
\end{proof}

\begin{prop} \label{pro:ij}
    Let $G$ be a graph and $H$ be its induced subgraph. Then for all $i, j \geq 0$ and $s \geq 1$, $$ \beta_{i,j}(S/L_H(d)^s) \leq \beta_{i,j}(S/L_G(d)^s).$$
\end{prop}
\begin{proof}
    Let $H$ be an induced subgraph of $G$ and $S_H=k[x_{ij} \mid i \in V(H), \text{ and } j \in [d]]$. First, we claim that $L_H(d)^s = L_G(d)^s \cap S_H$ for all $s \geq 1$, where $L_H(d)$ is a LSS-ideal of $H$ in $S_H$. We have $L_H(d)^s\subseteq L_G(d)^s\cap S_H$, since generators of $L_H(d)^s$ are contained in $L_G(d)^s$. For other side inclusion, consider the following map $\phi : S \rightarrow S_H$ by setting $\phi(x_{ij})=0$, if $x_{ij}\notin V(H)$ and $\phi(x_{ij})=x_{ij}$, if $x_{ij}\in V(H)$. Let $g=\sum_{e_1,\ldots,e_s \in E(G)}r_{e_1,\ldots,e_s}f_{e_1}^{(d)}\cdots f_{e_s}^{(d)} \in L_G(d)^s$, where $r_{e_1,\ldots,e_s} \in S$. Note that $\phi(g) = g$, if $g\in S_H$. Thus, we get
    \begin{equation*}
        \begin{split}
            g &= \sum_{e_1,\ldots,e_s \in E(G)}\phi(r_{e_1,\ldots,e_s})\phi(f_{e_1}^{(d)}\cdots f_{e_s}^{(d)}), \\
            &= \sum_{e_1,\ldots,e_s \in E(H)}\phi(r_{e_1,\ldots,e_s})f_{e_1}^{(d)}\cdots f_{e_s}^{(d)}.
        \end{split}
    \end{equation*}
    Therefore, $g \in L_H(d)^s$.  Now, we claim that ${S_H}/{L_H(d)^s}$ is an algebra retract of ${S}/{L_G(d)^s}$. Then the statement follows from \cite[Corollary 2.5]{OHH2000}. Consider, ${S_H}/{L_H(d)^s} \xhookrightarrow{i} {S}/{L_G(d)^s} \xrightarrow{\bar{\phi}} {S_H}/{L_H(d)^s}$, where $\bar{\phi}$ is an induced by the map $\phi$. Then one can see that $\bar{\phi} \circ i$ is identity on ${S_H}/{L_H(d)^s}$ and hence the claim.
\end{proof}

\begin{notation} \label{maxind}
    Let $d\geq 3$ and $G$ be a graph. We define two invariants of $G$ in the following way:
    \begin{enumerate}[(a)]
        \item $\mathfrak{t}(G) = \max\{|E(H)| \mid H$ is an induced subgraph of  $G$ such that $H$ is a forest with $\Delta(H) \leq d\}$.
        \item $\mathfrak{u}(G) = \max\{|E(H)| \mid H$ is an induced subgraph of  $G$ such that  $H$  is a unicyclic graph with $\Delta(H) \leq d\}$.
    \end{enumerate}   
\end{notation}

\begin{cor} \label{Cor:lower}
    Let $G$ be a graph. Then one has $$\reg(S/L_G(d)^s) \geq 2(s-1)+\max\{\mathfrak{t}(G),\mathfrak{u}(G)\},$$ (see Notation \ref{maxind}) for all $s\geq 1$. 
\end{cor}

\begin{proof}
    The assertion follows from Proposition \ref{rem:CIreg} and Proposition \ref{pro:ij}.
\end{proof}

%\textcolor{red}{The following results gives bounds for the regularity of powers of %LSS-ideals associated to trees and some specific graphs.}

\begin{thm}\label{thm4:t}
    Let $G$ be a tree on $[n]$. If $L_G(d)$ is an almost complete intersection, then for all $s\geq 1$, one has 
    $$2s + n-4 \leq \reg (S/(L_G(d))^s) \leq 2(s-1) + \reg (S/L_{G}(d).$$
    %Let $H_1$ and $H_2$ be trees with $\Delta(H_1) \leq d$, $\Delta(H_2) \leq d$. Let $G$ be a tree on $[n+1]$ such that $L_G(d)$ is not a complete intersection. If $G$ is obtained by adding an edge $e$ between two vertices of two trees $H_1$ and $H_2$ then for all $s\geq 1$, we have $$2s + n-3 \leq \reg S/(L_G(d))^s \leq 2(s-1) + \max\{\reg S/(L_{G}), n-2\}.$$
\end{thm}
%\begin{proof}
%    First, we show the lower bound. Suppose $G$ is a tree such that $L_G(d)$ is an almost complete intersection. From Theorem \ref{ACI:tree} it follows that $G$ is obtained by adding an edge $e$ between two complete intersection trees. Thus $G\setminus e$ is a complete intersection. Therefore, one has $\mathfrak{t}(G)=n-2$. From Corollary \ref{Cor:lower} it follows that $2s + n-4 \leq \reg (S/(L_G(d))^s)$, for all $s\geq 1$. Hence we have a desired lower bound. 

  %  Note that almost complete intersection ideals are $d$-sequence. Also, $L_G(d)$ is generated in degree $2$. Therefore, the upper bound follows directly from Remark \ref{Reg.upper}.
%\end{proof}

\begin{proof}
    Suppose $G$ is a tree with $L_G(d)$ being an almost complete intersection. From Theorem \ref{ACI:tree},  $G$ is obtained by adding an edge $e$ between two complete intersection trees. Thus, $G\setminus e$ is a complete intersection and so $\mathfrak{t}(G)=n-2$. From Corollary \ref{Cor:lower}, it then follows that, $2s + n-4 \leq \reg (S/(L_G(d))^s)$, for all $s\geq 1$. Hence, we have the desired lower bound. 

    Now, since an almost complete intersection ideal is generated by a $d$-sequence and $L_G(d)$ is generated in degree $2$, the upper bound follows from Remark \ref{Reg.upper}.
\end{proof}

\begin{thm} \label{thm4:ub}
Let $G$ be a connected graph on $[n]$. Let $T$ be a tree with $\Delta(T) \leq d$ and $V(T)=V(G)$, $U_1$ and $U_2$ be unicyclic graphs with with $\Delta(U_i) \leq d$, and $V(U_i)=V(G)$ for $i=1,2$, and $B$ be a bicyclic graph with $\Delta(B) \leq d$ and $V(B)=V(G)$. If $G$ is a connected graph on $[n]$ of one of the following forms: 
    \begin{enumerate}
        \item $G$ is obtained by adding an edge between two vertices of $T$ and $d \geq 3$;
        \item $G$ is obtained by adding an edge between a vertex of $T$ and a vertex of  $U_1$ and $d \geq 3$;
        \item $G$ is obtained by adding an edge between two vertices of $U_1$ and $d \geq 4$;
        \item $G$ is obtained by adding an edge between a vertex of $U_1$ and a vertex of $U_2$ and $d \geq 4$;
        \item  $G$ is obtained by adding an edge between a vertex of $T$ and a vertex of $B$ and $d \geq 4$.
    \end{enumerate}
    Then for all $s\geq 1$, one has 
    $$2s + n-3 \leq \reg (S/(L_G(d))^s) \leq 2(s-1) + \max\{\reg (S/L_{G}(d)), n-1\}.$$
\end{thm}
\begin{proof}
    From Theorem \ref{ACI:unicyclic} and Theorem \ref{ACI:bicyclic}, it follows that $G$ is obtained by adding an edge to a graph whose LSS-ideal is a complete intersection ideal. Also, $L_G(d)$ is generated in degree $2$. Therefore, the lower and the upper bounds follow from Corollary \ref{Cor:lower} and Remark \ref{Reg.upper}, respectively.
\end{proof}

\vspace{2mm}
%\section{Koszulness of quotients of LSS-ideals}


\begin{thebibliography}{99}
{
\addcontentsline{toc}{chapter}{References}

\bibitem{AN} M. Amalore Nambi, and N. Kumar: $d$-sequence edge binomials, and regularity of powers of binomial edge ideals of trees. \emph{J. Algebra Appl}. {\bf 23} (2024), no. 10, pp.2450154.

\bibitem{ANcycle} M. Amalore Nambi, and N. Kumar: \newblock{Regularity of powers of d-sequence (parity) binomial edge ideals of unicycle graphs.} {\emph{Comm. Algebra}}, {\bf 52} (2024) no. 6, 2598–2615.

\bibitem{ANpmd} M. Amalore Nambi, and N. Kumar: \newblock{On Positive Matching Decomposition Conjectures of Hypergraphs.} {arXiv:$2309.15424$v$2$.}


%\bibitem{AE} L. Avramov, and D. Eisenbud: \newblock Regularity of modules over a Koszul algebra. \emph{J. Algebra} {\bf 153} (1992), no. 1, {85-90}.


%\bibitem{AP} L. Avramov, and I. Peeva: \newblock Finite regularity and Koszul algebras. \emph{Amer.~J.~Math.} {\bf 123}  (2001), no. 2, 275-281.


%\bibitem{edge_survey} Banerjee, Arindam and Beyarslan, Selvi Kara and Huy T\`ai, H\`a: Regularity of edge ideals and their powers. \emph{Advances in Algebra.}, Springer, Cham, {\bf 277} (2019), 17–52. 


\bibitem{BHT15} S. Beyarslan, T. H. H\`a, and T. N. Trung: \newblock{Regularity of powers of forests and cycles.} \emph{J. Algebraic Combin.} {\bf 42} (2015), no.~4, 1077-1095.

\bibitem{Bruns&Conca} W.Bruns, A.Conca: \newblock Gr\"obner bases and determinantal ideals. \emph{Commutative Algebra, Singularities and
Computer Algebra}, Springer Netherlands (2003), 9--66.



%\bibitem{B2018} Bolognini, Davide and Macchia, Antonio and Strazzanti, Francesco: Binomial edge ideals of bipartite graphs. \emph{European J. Combin.} {\bf 70} (2018), 1–25. 


%\bibitem{CNR} A. Conca, E. De Negri, and M. E. Rossi: \newblock Koszul algebras and regularity. \emph{Commutative algebra}, (2013), {285-315}.


\bibitem{CW2019} A. Conca, and V. Welker: \newblock Lov\'{a}sz-{S}aks-{S}chrijver ideals and coordinate sections of determinantal varieties. \emph{Algebra Number Theory} {\bf 13} (2019), {no.~2}, {455-484}. 


\bibitem{CHT1999} S. Cutkosky, J. Herzog, and N. V. Trung: \newblock Asymptotic behaviour of the Castelnuovo-Mumford regularity. \emph{Compositio Math.} {\bf 118} (1999), no.~3, 243–261.


\bibitem{VRT2021}  V. Ene, G. Rinaldo, and N. Terai: Powers of binomial edge ideal with quadratic Gr\"{o}bner bases. \emph{Nagoya Math. J.} (2021), 1–23. 

\bibitem{FG2022} M. D. G. Farrokhi, S. Gharakhloo, and A. A. Yazdan Pour: \newblock Positive matching decompositions of graphs. \emph{Discrete Appl. Math.} {\bf 320} (2022), 311-323. 

%\bibitem{koszul} R. Fr\"{o}berg, and C. L\"{o}fwal: \newblock Koszul homology and {L}ie algebras with application to generic forms and points. \emph{Homology Homotopy Appl.}, {\bf 4} (2002),  {227--258}. 


%\bibitem{Froberg} R. Fr\"{o}berg: \newblock Koszul algebras. \emph{Lectures Notes in Pure and Applied Mathematics} {\bf 205} Marcel Dekker, (1999).

\bibitem{GW2023} S. Gharakhloo, and V. Welker: \newblock{Hypergraph LSS-ideals and coordinate sections of symmetric tensors.} \emph{Comm. Algebra} (2023), {1-11}.

\bibitem{M2} D. R. Grayson, and M. E. Stillman: Macaulay2, a software system for research in algebraic geometry, Available at \url{http://www.math.uiuc.edu/Macaulay2/}.

\bibitem{HH} J. Herzog, T. Hibi, F. Hreinsd\'{o}ttir, T. Kahle, and J. Rauh: Binomial edge ideals and conditional independence statements. \emph{Adv. in Appl. Math.} {\bf 45} (2010), no.~3, 317–333.

\bibitem{HMSW} J. Herzog, A. Macchia, S. Saeedi Madani, V. Welker: \newblock{On the ideal of orthogonal representations of a graph in {$\mathbb{R}^2$}.} {Adv. in Appl. Math.} {\bf 71} (2015), 146-173.



\bibitem{M1989} H. Matsumura: \newblock{Commutative ring theory.} {Cambridge Studies in Advanced Mathematics.} {Cambridge University Press, Cambridge} {\bf 8} (1986). {Translated from the Japanese by M. Reid.}




\bibitem{JAR2019} A. V. Jayanthan,  A. Kumar, and  R. Sarkar: \newblock Regularity of powers of quadratic sequences with applications to binomial ideals. \emph{J. Algebra} {\bf 564} (2020),  {98-118}.



\bibitem{V2000} V. Kodiyalam: \newblock Asymptotic behaviour of {C}astelnuovo-{M}umford regularity. \emph{Proc. Amer. Math. Soc.} {\bf 128} no.~2 (2000), 407–411. 


\bibitem{AK2021} A. Kumar: \newblock Lov\'{a}sz-{S}aks-{S}chrijver ideals and parity binomial edge ideals of graphs. \emph{European J. Combin.} {\bf 93} (2021), {Paper No. 103274, 19}.


\bibitem{A2021reg} A. Kumar: \newblock Regularity of parity binomial edge ideals. \emph{Proc. Amer. Math. Soc.} {\bf 149} (2021), no.~7, {2727--2737}. 

\bibitem{EK} E. Kunz: \newblock Almost complete intersections are not {G}orenstein rings. \emph{J. Algebra} {\bf 28} (1974), {111-115}.

\bibitem{LSS89} L. Lov\'{a}sz and M. Saks and A. Schrijver: \newblock{Orthogonal representations and connectivity of graphs.} \emph{Linear Algebra Appl.} {\bf 114/115} (1989), 439-454.


%\bibitem{CR20} C. Mascia, and G. Rinaldo: \newblock Krull dimension and regularity of binomial edge ideals of block graphs. \emph{J. Algebra Appl.} {\bf 19} (2020), no.~7, 2050133, 17pp. 


\bibitem{MM} K. Matsuda, and S. Murai: \newblock{Regularity bounds for binomial edge ideals.} \emph{J. Commut. Algebra} {\bf 5} (2013), no.~1, 141-149. 



\bibitem{OHH2000} H. Ohsugi, and J. Herzog, and T. Hibi: \newblock Combinatorial pure subrings. \emph{Osaka J. Math.} {\bf 37} (2000), {no.~3}, {745-757}.

\bibitem{O2011} M. Ohtani: Graphs and ideals generated by some 2-minors. \emph{Comm. Algebra} {\bf 39} (2011), no.~3, 905–917. 

%\bibitem{Priddy} S. B. Priddy: \newblock{Koszul resolutions.} \emph{Trans. Amer. Math. Soc.} {\bf 152} (1970), 39-60.

\bibitem{reg-seq} J. Saha, I. Sengupta, G. Tripathi: Primary decomposition and normality
of certain determinantal ideals. \emph{Proc. Indian Acad. Sci. Math. Sci.}, {\bf 129} (2019), no.~4, 1--10.


\bibitem{SZ} Y. Shen, and G. Zhu: Regularity of powers of (parity) binomial edge ideals. \emph{J. Algebraic Combin.} {\bf 57} (2022), {no.~1},  {75--100}. 


}
\end{thebibliography}
\end{document}